\documentclass[12pt]{amsart}
\usepackage{amssymb,amsmath,mathrsfs}
\usepackage[ps2pdf,colorlinks=true,urlcolor=blue,
citecolor=red,linkcolor=blue,linktocpage,pdfpagelabels,bookmarksnumbered,bookmarksopen]{hyperref}
\usepackage{epsfig,graphicx,color,mathrsfs}
\usepackage[english]{babel}

\newcommand{\N}{{\mathbb N}}

\newcommand{\R}{{\mathbb R}}

\newcommand{\be}{\begin{equation}}
\newcommand{\ee}{\end{equation}}

\numberwithin{equation}{section}
\newtheorem{theorem}{Theorem}[section]
\newtheorem{proposition}[theorem]{Proposition}
\newtheorem{corollary}[theorem]{Corollary}
\newtheorem{lemma}[theorem]{Lemma}
\newtheorem{definition}[theorem]{Definition}
\theoremstyle{definition}
\newtheorem{remark}[theorem]{Remark}

\newcommand{\brm}{\begin{remark}\rm}
\newcommand{\erm}{\end{remark}}
\newcommand{\brms}{\begin{remark}\rm}
\newcommand{\erms}{\end{remark}}
\newcommand{\bte}{\begin{theorem}}
\newcommand{\ete}{\end{theorem}}
\newcommand{\bpr}{\begin{proposition}}
\newcommand{\epr}{\end{proposition}}
\newcommand{\ble}{\begin{lemma}}
\newcommand{\ele}{\end{lemma}}
\newcommand{\beq}{\begin{equation}}
\newcommand{\eeq}{\end{equation}}
\newcommand{\bdm}{\begin{displaymath}}
\newcommand{\edm}{\end{displaymath}}
\numberwithin{equation}{section}

\newcommand{\bos}{\begin{remark}\rm}
\newcommand{\eos}{\end{remark}}

\newcommand{\ben}{\begin{enumerate}}
\newcommand{\een}{\end{enumerate}}

\title[Generalized Polya-Szeg\"o inequality]{Generalized 
Polya-Szeg\"o inequality}

\author[H.\ Hajaiej]{Hichem Hajaiej}

\thanks{Ipeit (Institut pr\'eparatoire aux \'etudes d'ing\'enieur de Tunis)
 2, Rue Jawaher Lel Nahru -
1089 Montfleury - Tunis,
 Tunisie.
E-mail: {\em hichem.hajaiej@gmail.com}}

\address{Ipeit (Institut pr\'eparatoire aux \'etudes d'ing\'enieur de Tunis)
\newline\indent
 2, Rue Jawaher Lel Nahru -
1089 Montfleury - tunis
 Tunisie}

\subjclass[2000]{46E35, 26B25, 26B99, 47B38}
\keywords{Generalized Polya-Szeg\"o inequalities, identity results, radial symmetry,
non-compact minimization problems}
\begin{document}

\begin{abstract}
We generalize Polya-Szeg\"o inequality to integrands depending on $u$ and its
gradient. Under minimal additional assumptions, we establish equality cases in this
generalized inequality. 
\end{abstract}
\maketitle

\medskip
\begin{center}
\begin{minipage}{11cm}
\footnotesize
\tableofcontents
\end{minipage}
\end{center}

%%%%%%%%%%%%%%%%%%%%%%%%%%%%%%%%%%%%%%
\medskip

\section{Introduction}

The Polya-Szeg\"o inequality asserts that the $L^2$ norm of the gradient of a positive function $u$
in $W^{1,p}(\R^N)$ cannot increase under Schwarz symmetrization,
\begin{equation}
	\label{classic}
\int_{\R^N}|\nabla u^*|^2dx\leq \int_{\R^N}|\nabla u|^2dx.
\end{equation}
The Schwarz rearrangement of $u$ is denoted here by $u^*$. Inequality~\eqref{classic} has numerous 
applications in physics. It was first used in 1945 by G.\ Polya and G. Szeg\"o
to prove that the capacity of a condenser diminishes or remains unchanged by 
applying the process of Schwarz symmetrization (see~\cite{polya1}).
Inequality~\eqref{classic} was also the key ingredients to show that, among
all bounded bodies with fixed measure, balls have the minimal capacity (see~\cite[Theorem 11.17]{liebloss}).
Finally~\eqref{classic} has also played a crucial role in the solution of the famous
Choquard's conjecture (see~\cite{lieb}). It is heavily connected to the isoperimetric inequality
and to Riesz-type rearrangement inequalities. Moreover, it turned out that~\eqref{classic} 
is extremely helpful in establishing the existence of ground states solutions of the nonlinear
Schr\"odinger equation
\begin{equation}
	\label{Sch}
\begin{cases}
{\rm i}\partial_t\Phi+\Delta\Phi+f(|x|,\Phi)=0  & \text{in $\R^N\times(0,\infty)$},\\
\Phi(x,0)=\Phi_0(x)  & \text{in $\R^N$}.
\end{cases}
\end{equation}
A ground state solution of equation~\eqref{Sch} is a positive  
solution to the following associated variational problem
\begin{equation}
\label{varPbSc}
\inf\left\{\frac{1}{2}\int_{\R^N}|\nabla u|^2dx-\int_{\R^N}F(|x|,u)dx:\, u\in H^1(\R^N),\,\,\|u\|_{L^2}=1\right\},
\end{equation}
where $F(|x|,s)$ is the primitive of $f(|x|,\cdot)$ with $F(|x|,0)=0$. Inequality~\eqref{classic}
together with the generalized Hardy-Littlewood inequality were crucial to prove that~\eqref{varPbSc}
admits a radial and radially decreasing solution.\ Furthermore, under appropriate regularity
assumptions on the nonlinearity $F$, there exists a Lagrange multiplier $\lambda$ such that any minimizer 
of~\eqref{varPbSc} is a solution of the following semi-linear elliptic PDE
$$
-\Delta u+f(|x|,u)+\lambda u=0,\quad\text{in $\R^N$}.
$$
We refer the reader to~\cite{hajstuampa} for a detailed analysis. The same approach applies 
to the more general quasi-linear PDE
$$
-\Delta_p u+f(|x|,u)+\lambda u=0,\quad\text{in $\R^N$}.
$$
where $\Delta_p u$ means ${\rm div}(|\nabla u|^{p-2}\nabla u)$, and we can derive similar 
properties of ground state solutions since~\eqref{classic} extends to gradients that
are in $L^p(\R^N)$ in place of $L^2(\R^N)$, namely
\begin{equation}
	\label{p-classic}
\int_{\R^N}|\nabla u^*|^pdx\leq \int_{\R^N}|\nabla u|^pdx.
\end{equation}
Due to the multitude of applications in physics, rearrangement inequalities like~\eqref{classic} and
\eqref{p-classic} have attracted a huge number of mathematicians from the middle of the
last century. Different approaches were built up to establish
these inequalities such as heat-kernel methods, slicing and cut-off techniques 
and two-point rearrangement.

A generalization of inequality~\eqref{p-classic} to suitable convex integrands $A:\R_+\to\R_+$,
\begin{equation}
	\label{A-classic}
\int_{\R^N}A(|\nabla u^*|)dx\leq \int_{\R^N} A(|\nabla u|)dx,
\end{equation}
was first established by Almgren and Lieb (see \cite{almgrenlieb}). Inequality~\eqref{A-classic}
is important in studying the continuity and discontinuity of Schwarz symmetrization
in Sobolev spaces (see e.g.\ \cite{almgrenlieb,burchard}). It also permits us to study
symmetry properties of variational problems involving integrals of type $\int_{\R^N}A(|\nabla u|)dx$.
Extensions of Polya-Szeg\"o inequality to more general operators of the form
$$
j(s,\xi)=b(s)A(|\xi|),\quad s\in\R,\,\xi\in\R^N,
$$
on bounded domains have been investigated by Kawohl, Mossino and Bandle. More precisely,
they proved that
\begin{equation}
	\label{PS-split}
\int_{\Omega^*}b(u^*)A(|\nabla u^*|)dx\leq \int_{\Omega} b(u)A(|\nabla u|)dx,
\end{equation}
where $\Omega^*$ denotes the ball in $\R^N$ centered at the origin having
the Lebesgue measure of $\Omega$, under suitably convexity, monotonicity
and growth assumptions (see e.g.\ \cite{bandle,kawohl,mossino}).
Numerous applications of~\eqref{PS-split} have been discussed in the above references.
In~\cite{tahraoui}, Tahraoui claimed that a general integrand $j(s,\xi)$ with appropriate properties
can be written in the form
$$
\sum_{i=1}^\infty b_i(s)A_i(|\xi|)+R_1(s)+R_2(\xi),\quad s\in\R,\,\xi\in\R^N,
$$
where $b_i$ and $A_i$ are such that inequality~\eqref{PS-split} holds. However, there are
some mistakes in~\cite{tahraoui} and we do not believe that this density type result holds true.
Until quite recently there were no results dealing with the generalized 
Polya-Szeg\"o inequality, namely
\begin{equation}
	\label{generalPS}
\int_{\Omega^*}j(u^*,|\nabla u^*|)dx\leq \int_{\Omega} j(u,|\nabla u|)dx.
\end{equation}
While writing down this paper we have learned about a very recent survey by 
F.\ Brock \cite{brock} who was able to prove~\eqref{generalPS} under continuity,
monotonicity, convexity and growth conditions.

Following a completely different approach, we prove~\eqref{generalPS}
without requiring any growth conditions on $j$. As it can be easily seen it is important to
drop these conditions to the able to cover some relevant applications.
Our approach is based upon a suitable approximation of the Schwarz symmetrized
$u^*$ of a function $u$. More precisely, if  
$(H_n)_{n\geq 1}$ is a dense sequence in the set of closed half spaces $H$ containing $0$
and $u\in L^{p}_+(\R^N)$, there exists a 
sequence $(u_n)$ consisting of iterated polarizations of the $H_n$s which
converges to $u^*$ in $L^p(\R^N)$ (see~\cite{explicit,JvS}). 
On the other hand, a straightforward computation shows that
$$
	\|\nabla u\|_{L^p(\R^N)}=\|\nabla u_0\|_{L^p(\R^N)}=\cdots=\|\nabla u_n\|_{L^p(\R^N)},
	\quad\text{for all $n\in\N$}.
$$
By combining these properties with the weak lower semicontinuity of the functional
$J(u)=\int j(u,|\nabla u|)dx$ enable us to conclude (see Theorem~\ref{gpz}). Note that~\eqref{A-classic}
was proved using coarea formula; however this approach does not apply to integrands
depending both on $u$ and its gradient since one has to apply simultaneously the coarea formula
to $|\nabla u|$ and to decompose $u$ with the Layer-Cake principle.

Notice that Brock's method is based on an intermediate maximization problem and cannot yield to the establishment
of equality cases. Our approximation approach was also fruitful in determining the relationship between
$u$ and $u^*$ such that
\begin{equation}
	\label{EqgeneralPS}
\int_{\R^N}j(u^*,|\nabla u^*|)dx=\int_{\R^N} j(u,|\nabla u|)dx.
\end{equation}
Indeed, under very general conditions on $j$, we prove that~\eqref{EqgeneralPS}
is equivalent to 
$$
\int_{\R^N}|\nabla u^*|^pdx=\int_{\R^N}|\nabla u|^pdx.
$$ 
For $j(\xi)=|\xi|^p$, identity cases were completely studied 
in the breakthrough paper of Brothers and Ziemer~\cite{bzim}.

\vskip15pt
\noindent
The paper is organized as follows.
\vskip4pt
\noindent
Section~\ref{prelimfacts} is dedicated to some preliminary stuff, especially the ones 
concerning the invariance of a class of functionals under polarization. 
These observations are crucial, in Section~\ref{PSineq}, to establish 
in a simple way the generalized Polya-Szeg\"o inequality.
\medskip

\vskip15pt
\begin{center}\textbf{Notations.}\end{center}
\begin{enumerate}
\item For $N\in\N$, $N\geq 1$, we denote by $|\cdot|$ the euclidean norm in $\R^N$.
\item $\R_+$ (resp.\ $\R_-$) is the set of positive (resp.\ negative) real values.
\item $\mu$ denotes the Lebesgue measure in $\R^N$.
\item $M(\R^N)$ is the set of measurable functions in $\R^N$.
\item For $p>1$ we denote by $L^p(\R^N)$ the space of $f$ in $M(\R^N)$
with $\int_{\R^N}|f|^pdx<\infty$.
\item The norm $(\int_{\R^N}|f|^pdx)^{1/p}$ in $L^p(\R^N)$ is denoted by $\|\cdot\|_p$.
\item For $p>1$ we denote by $W^{1,p}(\R^N)$ the Sobolev space of functions $f$ in $L^p(\R^N)$
having generalized partial derivatives $D_if$ in $L^p(\R^N)$, for $i=1,\dots, N$.
\item $D^{1,p}(\R^N)$ is the space of measurable functions whose gradient is in $L^p(\R^N)$.
\item $L^{p}_+(\R^N)$ is the cone of positive functions of $L^{p}(\R^N)$.
\item $W^{1,p}_+(\R^N)$ is the cone of positive functions of $W^{1,p}(\R^N)$.
\item For $R>0$, $B(0,R)$ is the ball in $\R^N$ centered at zero with radius $R$.
\end{enumerate}
\medskip

\section{Preliminary stuff}
\label{prelimfacts}
In the following $H$ will design a closed half-space of $\R^N$ containing the origin, $0_{\R^N}\in H$.
We denote by ${\mathcal H}$ the set of closed half-spaces of $\R^N$ containing the origin. We shall
equip ${\mathcal H}$ with a topology ensuring that $H_n\to H$ as $n\to\infty$ if there is a sequence of
isometries $i_n:\R^N\to\R^N$ such that $H_n=i_n(H)$ and $i_n$ converges to the identity as $n\to\infty$.
\vskip4pt
We first recall some basic notions. For more details, we refer the reader to~\cite{burchhaj}.

\begin{definition}
	A reflection $\sigma:\R^N\to\R^N$ with respect to $H$
	is an isometry such that the following properties hold
	\begin{enumerate}
		\item $\sigma\circ\sigma (x)=x$, for all $x\in\R^N$;
		\item the fixed point set of $\sigma$ separates $\R^N$ in $H$ and $\R^N\setminus H$ (interchanged by $\sigma$);
		\item $|x-y|<|x-\sigma(y)|$, for all $x,y\in H$.
	\end{enumerate}
	Given $x\in\R^N$, the reflected point $\sigma_H(x)$ will also be denoted by $x^H$.
\end{definition}

\begin{definition}
Let $H$ be a given half-space in $\R^N$.
The two-point rearrangement (or polarization) of a nonnegative real valued function
$u:\R^N\to\R_+$ with respect to a given reflection $\sigma_H$ (with respect to $H$)
is defined as 
$$
u^H(x):=
\begin{cases}
	\max\{u(x),u(\sigma_H(x))\}, & \text{for $x\in H$}, \\
	\min\{u(x),u(\sigma_H(x))\}, & \text{for $x\in \R^N\setminus H$}.
\end{cases}
$$
\end{definition}

\begin{definition}
	We say that a nonnegative measurable function $u$ is symmetrizable if
$\mu(\{x\in\R^N: u(x)>t\})<\infty$ for all $t>0$. The space of symmetrizable 
functions is denoted by $F_N$ and, of course, $L^p_+(\R^N)\subset F_N$.
Also, two functions $u,v$ are said to be equimeasurable (and we shall write $u\sim v$) when 
$$
\mu(\{x\in\R^N: u(x)>t\})=\mu(\{x\in\R^N: v(x)>t\}),
$$ 
for all $t>0$. 
\end{definition}

\begin{definition}
	For a given $u$ in $F_N$, the Schwarz symmetrization $u^*$ of $u$ is the unique
	function with the following properties (see e.g.~\cite{hajstu})
	\begin{enumerate}
		\item $u$ and $u^*$ are equimeasurable;
		\item $u^*(x)=h(|x|)$, where $h:(0,\infty)\to\R_+$ is a continuous and decreasing function.
	\end{enumerate}
	In particular, $u$, $u^H$ and $u^*$ are all equimeasurable functions (see e.g.~\cite{bae}).
	\end{definition}

\begin{lemma}
	\label{concretelem}
	Let $u\in W^{1,p}_+(\R^N)$ and let $H$ be a given half-space. Then $u^H\in W^{1,p}_+(\R^N)$ and, setting
	$$
	v(x):=u(x^H),\quad w(x):=u^H(x^H),\qquad x\in\R^N,
	$$
	the following facts hold:
	\begin{enumerate}
		\item We have
	\begin{align*}
		\nabla u^H(x)&=
		\begin{cases}
			\nabla u(x) & \text{for $x\in \{u>v\}\cap H$}, \\
			\nabla v(x) & \text{for $x\in \{u\leq v\}\cap H$}, \\
		\end{cases} \\
		\nabla w(x)&=
		\begin{cases}
			\nabla v(x) & \text{for $x\in \{u>v\}\cap H$}, \\
			\nabla u(x) & \text{for $x\in \{u\leq v\}\cap H$}. \\
		\end{cases}
	\end{align*}
	\item For all $i=1,\dots,N$ and $p\in (1,\infty)$, we have
	\begin{equation}
		\label{dis4disocase}
		\|D_i u^H\|_{L^{p}(\R^N)}= 	\|D_i u\|_{L^{p}(\R^N)}.
	\end{equation}
	\item Let $j:[0,\infty)\times [0,\infty)\to\R$ be a Borel measurable function. Then
 \begin{equation}
	\label{fundineq}
		\int_{\R^N}j(u,|\nabla u|)dx=\int_{\R^N}j(u^H,|\nabla u^H|)dx,
	\end{equation}
	provided that $0\in H$ and that both integrals are finite.
\end{enumerate}
\end{lemma}
\begin{proof}
	Observing that, for all $x\in H$, we have
	$$
	u^H(x)=v(x)+(u(x)-v(x))^+,\qquad
	w(x)=u(x)-(u(x)-v(x))^+,
	$$
	in light of~\cite[Corollary 6.18]{liebloss} it follows that
	$v,w$ belong to $W^{1,p}_+(\R^N)$. Assertion (1) follows by a simple direct computation.
	Assertion (2) follows as a consequence of assertion (1). Concerning (3), writing $\sigma_H$ 
    as $\sigma_H(x)=x_0+Rx$, where $R$ is an orthogonal linear transformation,
	taking into account that $|{\rm det}\,R|=1$ and
	$$
	|\nabla v(x)|=|\nabla (u(\sigma_H(x)))|=|R(\nabla u(\sigma_H(x)))|=|(\nabla u)(\sigma_H(x))|, 
	$$ 
	we have
	\begin{align*}
		\int_{\R^N}j(u,|\nabla u|)dx & =\int_{H}j(u,|\nabla u|)dx+\int_{\R^N\setminus H}j(u,|\nabla u|)dx \\
&		=\int_{H}j(u,|\nabla u|)dx+\int_{H}j(u(\sigma_H(x)),|(\nabla u)(\sigma_H(x))|)dx \\
&		=\int_{H}j(u,|\nabla u|)dx+\int_{H}j(v,|\nabla v|)dx.
	\end{align*}
	In a similar fashion, we have
	\begin{align*}
		\int_{\R^N}j(u^H,|\nabla u^H|)dx & =\int_{H}j(u^H,|\nabla u^H|)dx+\int_{H}j(u^H(\sigma_H(x)),|(\nabla u^H)(\sigma_H(x))|)dx \\
&		=\int_{H}j(u^H,|\nabla u^H|)dx+\int_{H}j(w,|\nabla w|)dx \\
&		=\int_{\{u>v\}\cap H}j(u,|\nabla u|)dx+\int_{\{u>v\}\cap H}j(v,|\nabla v|)dx \\
&		+\int_{\{u\leq v\}\cap H}j(v,|\nabla v|)dx+\int_{\{u\leq v\}\cap H}j(u,|\nabla u|)dx \\
&		=\int_{H}j(u,|\nabla u|)dx+\int_{H}j(v,|\nabla v|)dx,
	\end{align*}
	which concludes the proof
\end{proof}
\medskip

\section{Generalized Polya-Szeg\"o inequality}
\label{PSineq}

The first main result of the paper is the following

\begin{theorem}
	\label{gpz}
Let $\varrho:[0,\infty)\times\R^N\to\R$ be a Borel measurable function. For any function $u\in W^{1,p}_+(\R^N)$,
let us set
$$
J(u)=\int_{\R^N} \varrho(u,\nabla u)dx.
$$
Moreover, let $(H_n)_{n\geq 1}$ be a dense sequence in the set of closed half spaces containing~$0_{\R^N}$. 
For $u\in W^{1,p}_+(\R^N)$, define a sequence $(u_n)$ by setting
$$
\begin{cases}
u_0=u  &\\
u_{n+1}=u_n^{H_1\ldots H_{n+1}}. &
\end{cases}
$$
Assume that the following conditions hold:
\begin{enumerate}
	\item  $$-\infty<J(u)<+\infty;$$
	\item 
	\begin{equation} 
	\label{monotonicity}
\liminf_n J(u_n)\leq J(u);
    \end{equation}
	\item if $(u_n)$ converges weakly to some $v$ in $W^{1,p}_+(\R^N)$, then 
	$$
	J(v)\leq\liminf_n J(u_n).
	$$
\end{enumerate}
Then
$$
J(u^*)\leq J(u).
$$
\end{theorem}

\begin{proof}
By the (explicit) approximation results contained in~\cite{explicit,JvS},
we know that $u_n\to u^*$ in $L^p(\R^N)$ as $n\to\infty$.
Moreover, by Lemma~\ref{concretelem} applied with $j(s,|\xi|)=|\xi|^p$, we have
\begin{equation}
	\label{uguagplapl}
	\|\nabla u\|_{L^p(\R^N)}=\|\nabla u_0\|_{L^p(\R^N)}=\cdots=\|\nabla u_n\|_{L^p(\R^N)},
	\quad\text{for all $n\in\N$}.
\end{equation}
In particular, up to a subsequence, $(u_n)$ is weakly convergent 
to some function $v$ in $W^{1,p}(\R^N)$. By uniqueness of the weak limit in $L^p(\R^N)$
one can easily check that $v=u^*$, namely $u_n\rightharpoonup u^*$ in $W^{1,p}(\R^N)$.
Hence, using assumption (3) and ~\eqref{monotonicity}, we have 
\begin{equation}
	J(u^*)\leq \liminf_n J(u_n)\leq J(u),
\end{equation}
concluding the proof.
\end{proof}

\begin{remark}
	A quite large class of functionals $J$ which satisfy assumption~\eqref{monotonicity} 
	of the previous Theorem is provided by Lemma~\ref{concretelem}.
\end{remark}

\begin{corollary}
	\label{applgPSineq}
	Let $j:[0,\infty)\times[0,\infty)\to\R$ be a function satisfying the following assumptions:
	\begin{enumerate}
		\item $j(\cdot,t)$ is continuous for all $t\in [0,\infty)$;
		\item $j(s,\cdot)$ is convex for all $s\in [0,\infty)$ and continuous at zero;
		\item $j(s,\cdot)$ is nondecreasing for all $s\in [0,\infty)$.
	\end{enumerate}
	Then, for all function $u\in W^{1,p}_+(\R^N)$ such that
	$$
	\int_{\R^N} j(u,|\nabla u|)dx<\infty,
	$$
	we have
	$$
	\int_{\R^N} j(u^*,|\nabla u^*|)dx\leq \int_{\R^N} j(u,|\nabla u|)dx.
	$$
\end{corollary}
\begin{proof}
	The assumptions on $j$ imply that $\{\xi\mapsto j(s,|\xi|)\}$ is convex
	so that the weak lower semicontinuity assumption of Theorem~\ref{gpz} holds	
	(we refer the reader e.g.\ to the papers~\cite{ioffe1,ioffe2} by A.\ Ioffe).
	Also, assumption~\eqref{monotonicity} of Theorem~\ref{gpz} is provided by 
	means of Lemma~\ref{concretelem}.
\end{proof}

\begin{remark}
	In~\cite[Theorem 4.3]{brock}, F.\ Brock proved Corollary~\ref{applgPSineq}
	for Lipschitz functions having compact support. In order to prove the most
	interesting cases in the applications, the inequality has to hold for 
	functions $u$ in $W^{1,p}_+(\R^N)$. This forces him to assume some growth conditions
	of the Lagrangian $j$, for instance to assume that there exists a positive constant
	$K$ and $q\in [p,p^*]$ such that
	$$
	|j(s,|\xi|)|\leq K(s^q+|\xi|^p),\quad\text{for all $s\in\R_+$ and $\xi\in\R^N$}.
	$$
	By our approach, instead, can include integrands such as 
	$$
	j(s,|\xi|)=\frac{1}{2}(1+s^{2\alpha})|\xi|^p,\quad\text{for all $s\in\R_+$ and $\xi\in\R^N$},
	$$
	for some $\alpha>0$, which have meaningful physical 
	applications (for instance quasi-linear Schr\"odinger equations, see~\cite{liuwang}
	and references therein).
	We also stress that the approach of~\cite{brock} cannot yield the establishment
	of equality cases (see Theorem~\ref{idcases}).
\end{remark}

\begin{corollary}
	\label{finaleconcr2}
	Let $m\geq 1$ and $p_1,\dots,p_m\in(1,\infty)$. Then
	$$
	\sum_{i=1}^m\int_{\R^N}|D_i u^*|^{p_i}dx\leq\sum_{i=1}^m\int_{\R^N}|D_i u|^{p_i}dx,
	$$
	for all $u\in\bigcap_{i=1}^m W^{1,p_i}_+(\R^N)$.
\end{corollary}
\begin{proof}
	The assertion follows by a simple combination of Theorem~\ref{gpz}
	with inequality~\eqref{dis4disocase} of Lemma~\ref{concretelem}.
\end{proof}

\begin{theorem}
	\label{idcases}
	In addition to the assumptions of Theorem~\ref{gpz}, assume that 
\begin{equation}
	\label{strict}
\text{$J(u_n)\to J(u^*)$ as $n\to\infty$ implies that $u_n\to u^*$ in $D^{1,p}(\R^N)$ as $n\to\infty$}.
\end{equation}
Then
$$
J(u)=J(u^*)\,\,\Longrightarrow\,\, \|\nabla u\|_{L^p(\R^N)} =\|\nabla u^*\|_{L^p(\R^N)}.
$$
\end{theorem}
\begin{proof}
Assume that $J(u)=J(u^*)$. Then, by assumption~\eqref{monotonicity}, we obtain
$$
J(u^*)=\lim_n J(u_n)=J(u).
$$
In turn, by assumption, $u_n\to u^*$ in $D^{1,p}(\R^N)$ as $n\to\infty$. Then, taking 
the limit inside equalities~\eqref{uguagplapl}, we conclude the assertion.
\end{proof}

\begin{remark}
	Assume that $\{\xi\mapsto j(s,|\xi|)\}$ is strictly convex for any $s\in\R_+$ and
	there exists $\nu'>0$ such that $j(s,|\xi|)\geq\nu' |\xi|^p$ for all $s\in\R_+$ and $\xi\in\R^N$. 
	Then assumption~\eqref{strict} is fulfilled for $J(u)=\int_{\R^N}j(u,|\nabla u|)dx$. 
	We refer to~\cite[Section 3]{visintin}.
\end{remark}

\begin{remark}
	Equality cases of the type $\|\nabla u\|_{L^p(\R^N)}=\|\nabla u^*\|_{L^p(\R^N)}$
	have been completely characterized  in the breakthrough paper by Brothers and Ziemer~\cite{bzim}. 
\end{remark}

Let us now set
$$
\displaystyle{M={\rm esssup}_{\R^N} u={\rm esssup}_{\R^N} u^*},\qquad
C^*=\{x\in\R^N:\nabla u^*(x)=0\}.
$$

\begin{corollary}
	\label{identitcor}
	Assume that $\{\xi\mapsto j(s,|\xi|)\}$ is strictly convex and there exists a positive constant
	$\nu'$ such that
	$$
	j(s,|\xi|)\geq\nu' |\xi|^p,\quad\text{for all $s\in\R$ and $\xi\in\R^N$}.
$$	
	Moreover, assume that
$$
\int_{\R^N} j(u,|\nabla u|)dx=\int_{\R^N} j(u^*,|\nabla u^*|)dx,\quad
\mu(C^*\cap (u^*)^{-1}(0,M))=0.
$$
Then there exists $x_0\in\R^N$ such that
$$
u(x)=u^*(x-x_0),\quad\text{for all $x\in\R^N$},
$$
namely $u$ is radially symmetric after a translation in $\R^N$.
\end{corollary}
\begin{proof}
	It is sufficient to combine Theorem~\ref{idcases} with~\cite[Theorem 1.1]{bzim}.
\end{proof}


\medskip


\begin{thebibliography}{99}
	
\bibitem{almgrenlieb}
{\sc F.J.\	Almgren, E.H.\ Lieb}, 
Symmetric decreasing rearrangement is sometimes continuous,
{\em J.\ Amer.\ Math.\ Soc.} {\bf 2} (1989), 683--773. 
	
\bibitem{bae} 
{\sc A.\ Baernstein}, 
A unified approach to symmetrization, in: Partial Differential equations 
of elliptic type, eds, A.\ Alvino et al., 
Symposia matematica 35, Cambridge University Press 1995, 47--91.
 
\bibitem{bandle}
{\sc C.\ Bandle}, 
Isoperimetric inequalities and applications, 
Monographs and Studies in Math.\ Pitman, London, 1980.

\bibitem{BL1}
{\sc H.\ Berestycki, P.L.\ Lions}, 
Nonlinear scalar field equations. I. Existence of a ground state.
{\em Arch.\ Rational Mech.\ Anal.} {\bf 82} (1983), 313--345.

\bibitem{c2}
{\sc H.\ Brezis},
Analyse fonctionnelle, Th\'eorie et applications,
{\em Editions Masson}, 1984.

\bibitem{brock}
{\sc F.\ Brock}
Rearrangements and applications to symmetry problems in PDE.
Survey paper.

\bibitem{c5}
{\sc F.\ Brock},
Rearrangement inequalities \`a la Hardy-littlewood,
{\em J.\ Ineq.\ Appl.}, (2000), 309--320.

\bibitem{c4}
{\sc F.\ Brock, Y.\ Solynin},
An approach to symmetrization via polarization,
{\em Trans.\ Amer.\ Math.\ Soc.} {\bf 352}, (2000), 1759--1796.

\bibitem{brokhaj}
{\sc F.\ Brock, H.\ Hajaiej},
On the necessity of supermodularity in rearrangement inequalities, 
{\em preprint}.

\bibitem{bzim}
{\sc J.E.\ Brothers, W.P.\ Ziemer}, 
Minimal rearrangements of Sobolev functions,  
{\em J.\ Reine Angew.\ Math.}  {\bf 384} (1988), 153--179.

\bibitem{burchard}
{\sc A.\ Burchard},
Steiner symmetrization is continuous in $W^{1,p}$,
{\em Geom.\ Funct.\ Anal.} {\bf 7} (1997), 823--860. 

\bibitem{burchhaj} 
{\sc A.\ Burchard, H.\ Hajaiej},   
Rearrangement inequalities for functionals with monotone integrands.
{\em J.\ Functional Analysis} {\bf 233}, 561--582.

\bibitem{bjm}
{\sc J.\ Byeon, L.\ Jeanjean, M.\ Mari\c s},
Symmetry and monotonicity of least energy solutions,
{\em Calc.\ Var.\ Partial Differentil Equations}, 
in press. (DOI:10.1007/s0052 6-009-0238-1).

\bibitem{candeg}
{\sc A.\ Canino, M.\ Degiovanni}, 
Nonsmooth critical point theory and quasilinear elliptic equations. 
Topological methods in differential equations and inclusions (Montreal, PQ, 1994), 1--50,
{\em NATO Adv.\ Sci.\ Inst.\ Ser.\ C Math.\ Phys.\ Sci.} {\bf 472}, 
Kluwer Acad.\ Publ., Dordrecht, 1995. 

\bibitem{c6}
{\sc C.\ Draghici}, 
Rearrangement inequalities with applications to ratio of heat kernels,
{\em Potential Analysis 22}, (2005), 351--374.

\bibitem{fresh}
{\sc J.\ Frehse},
A note on the H\"older continuity of solutions of variational problems,
{\em Abh.\ Math.\ Sem.\ Univ.\ Hamburg} {\bf 43} (1975), 59--63.

\bibitem{explicit}
{\sc H.\ Hajaiej},
Explicit constructive approximation to symmetrization via iterated polarization,
dedicated to Al Baernstein at the occasion of his 70th birthday, 
{\em preprint}.

\bibitem{c7}
{\sc H.\ Hajaiej},
Cases of equality and strict inequality in the extended Hardy-Littlewood inequalities,
{\em Proc. Roy. Soc. Edinburgh} {\bf 135} (2005), 643--661.

\bibitem{hajstu}  
{\sc H.\ Hajaiej, C.A.\ Stuart},
Symmetrization inequalities for composition operators of Carath\'eodory type,
{\em Proc. London Math. Soc.} {\bf 87} (2003), 396--418. 

\bibitem{hajstuampa} 
{\sc H.\ Hajaiej, C.A.\ Stuart},
Existence and non-existence of Schwarz symmetric ground states for elliptic eigenvalue problems, 
{\em Ann.\ Mat.\ Pura Appl.} {\bf 184} (2005), 297--314.

\bibitem{ioffe1}
{\sc A.\ Ioffe},
On lower semicontinuity of integral functionals. I,
{\em SIAM J.\ Control Optimization}  {\bf 15} (1977), 521--538.

\bibitem{ioffe2}
{\sc A.\ Ioffe},
On lower semicontinuity of integral functionals. II,
{\em SIAM J.\ Control Optimization} {\bf 15}  (1977), 991--1000.

\bibitem{jjsq}
{\sc L.\ Jeanjean, M.\ Squassina}, 
Radial symmetry of least energy solutions for a class of quasi-linear elliptic equations,  
{\em Ann.\ Inst.\ H.\ Poincar\'e Anal.\ Non Lin\'eaire C}, to appear
(DOI: 10.1016/j.anihpc.2008.11.003)

\bibitem{kawohl}
{\sc B.\ Kawohl}, 
On rearrangements, symmetrization and maximum principles, 
Lecture Notes Math.\ {\bf 1150}, Springer, Berlin, 1985.

\bibitem{lieb}
{\sc E.H.\ Lieb},
Existence and uniqueness of the minimizing solution of Choquard's nonlinear equation,
{\em Studies in Appl.\ Math.} {\bf 57} (1976/77), 93--105. 

\bibitem{liebloss} 
{\sc E.H.\ Lieb, M.\ Loss},  
Analysis, second edition.
Graduate Studies in Mathematics, {\bf 14}
American mathematical society, 2001.

\bibitem{liuwang}
{\sc J.\ Liu, Y.\ Wang, Z.Q.\ Wang}, 
Solutions for quasi-linear Schr\"odinger equations via the Nehari method,
{\em Comm.\ Partial Differential Equations} {\bf 29} (2004), 879--901. 

\bibitem{Ma}
{\sc M.\ Mari\c s}, 
On the symmetry of minimizers, 
{\em Arch.\ Rat.\ Mech.\ Anal.} {\bf 192} (2009), 311--330.

\bibitem{mossino}
{\sc J.\ Mossino},
In\'egalit\'es isop\'erim\'etriques et applications en physique,
Hermann, Paris, 1984. 

\bibitem{polya1}
{\sc G.\ Polya, G.\ Szeg\"o}, 
Inequalities for the capacity of a condenser,  
{\em Amer.\ J.\ Math.} {\bf 67} (1945), 1--32.

\bibitem{serrin}
{\sc J. Serrin},
Local behavior of solutions of quasi-linear equations,
{\em Acta Math.} {\bf 111} (1964), 247--302.

\bibitem{sirakov}
{\sc B.\ Sirakov}, 
Least energy solitary waves for a system of nonlinear Schr\"odinger 
equations in $\R^n$, 
{\em Commun.\ Math.\ Phys.} {\bf 271} (2007), 199--221.

\bibitem{squastoul}
{\sc M.\ Squassina},
Weak solutions to general Euler's equations via nonsmooth critical point theory,
{\em Ann.\ Fac.\ Sci.\ Toulouse Math.} {\bf 9} (2000), 113--131.

\bibitem{stuarbbelow}
{\sc C.A.\ Stuart},
Bifurcation for Dirichlet problems without eigenvalues,
{\em Proc.\ London Math.\ Soc.} {\bf 45} (1982), 169--192. 

\bibitem{tahraoui}
{\sc R.\ Tahraoui}, 
Symmetrization inequalities, 
{\em Nonlinear Anal.} {\bf 27} (1996), 933--955. 
Corrigendum in {\em Nonlinear Anal.} {\bf 39} (2000), 535. 

\bibitem{tolks}
{\sc P.\ Tolksdorf},
Regularity for a more general class of quasilinear elliptic equations,
{\em J.\ Differential Equations} {\bf 51} (1984), 126--150.

\bibitem{troy}
{\sc W.C.\ Troy}, 
Symmetry properties in systems of semilinear elliptic equations,
{\em J.\ Differential Equations} {\bf 42} (1981), 400--413. 

\bibitem{JvS}
{\sc J. Van Schaftingen},
Explicit approximation of the symmetric rearrangement by polarizations,
{\em preprint}.

\bibitem{visintin}
{\sc A.\ Visintin},
Strong convergence results related to strict convexity, 
{\em Comm.\ Partial Differential Equations} {\bf 9} (1984), 439--466.

\end{thebibliography}
\end{document}